\documentclass{article} 

\usepackage{hyperref}

\usepackage[thmmarks]{ntheorem}
\usepackage[tbtags]{amsmath}
\usepackage{amsfonts}
\usepackage{enumerate}

\usepackage{caption}
\usepackage{subcaption}
\usepackage{amssymb}
\usepackage{tikz}       
\usetikzlibrary{calc,arrows}

{   
	\theoremstyle{nonumberplain}
	\theoremheaderfont{\bfseries}
	\theorembodyfont{\normalfont}
	\theoremsymbol{$\square$}
	\newtheorem{proof}{Proof}
}

\allowdisplaybreaks[3]

\newtheorem{definition}{Definition}[section]
\newtheorem{theorem}{Theorem}[section]
\newtheorem{corollary}[theorem]{Corollary}
\newtheorem{lemma}[theorem]{Lemma}
\newtheorem{proposition}[theorem]{Proposition}

\usepackage[bottom=25mm,left=25mm,right=25mm,top=25mm]{geometry}

\begin{document}
	
	
	\title{Structure and enumeration results of matchable Lucas cubes
	}
	
	\author{Xu Wang, Xuxu Zhao and Haiyuan Yao\footnote{Corresponding author.}%
		\\ {\footnotesize College of Mathematics and Statistics, Northwest Normal University, Lanzhou 730070, PR China}}
	
	\date{}
	
	\maketitle

\begin{abstract}
	
A lucasene is a hexagon chain that is similar to a fibonaccene, an $L$-fence is a poset the Hasse diagram of which is isomorphic to the directed inner dual graph of the corresponding lucasene.
A new class of cubes, which named after matchable Lucas cubes according to the number of its vertices (or elements), are a series of directed or undirected Hasse diagrams of filter lattices of $L$-fences.
The basic properties and several classes of polynomials, e.g.\ rank generating functions, cube polynomials and degree sequence polynomials, of matchable Lucas cubes are obtained.
Some special conclusions on binomial coefficients and Lucas triangle are given.


\textbf{2010 AMS Subj. Class.:} 11B39, 05C70, 06D05, 06A07 
\end{abstract}

\section{Introduction}

The $Z$-transformation graph (also called resonance graph) is introduced independently by Gr\"undler \cite{aGruen82}, Zhang et al.\ \cite{aZhangGC88b}, Randi\'c \cite{aRandi97} and Fournier \cite{aFourn03},
and widely studied by Zhang and Zhang \cite{aZhangZ99,aZhangZ00}, Kalva\v{z}ar et al.\ \cite{aKlavzVZ03,aKlavzZ05,aKlavzZB02}, Lam and Zhang \cite{aLamZ03}, Zhang et al.\ \cite{aZhang06,aZhangOY09,aZhangZY04c} and {\v{Z}igert Pleter\v{s}ek} \cite{aZiger17,aZigerB13b}. Recently, Zhang et al.\ \cite{aZhangYY14} introduced the concept of matchable distributive lattice and got some consequences on matchable distributive lattices, Yao and Zhang \cite{aYaoZ15} obtained some results on non-matchable distributive lattices with a cut-element.

The Fibonacci cubes $\Gamma_n$ \cite{aHsu93} are defined by Hsu, and Klav\v{z}ar and {\v{Z}igert Pleter\v{s}ek} \cite{aKlavzZ05} found that the Fibonacci cubes are the resonance graphs of fibonaccenes.
The Lucas cubes or Lucas lattices $\Lambda_n$ \cite{aMunarCZ01,Zagag01} are defined similarly.
In addition, the rank gererating functions \cite{aMunarZ02b}, the cube polynomials \cite{aKlavM12,aSaygE18}, the maximal cube polynomials \cite{aMolla12}, the disjoint cube polynomials \cite{aGraviMSZ15,aSaygE16} and the degree sequence polynomials \cite{aKlavzMP11} of Fibonacci and Lucas cubes are studied.
And Klav\v{z}ar have a survey \cite{aKlavz13} on Fibonacci cubes.
In addition, Yao and Zhang \cite{aYaoZ15} obtained the matchabilities of Fibonacci and Lucas cubes.

The structure of this paper is as follows.
The main concept in the paper, \emph{matchable Lucas cube}, is introduced by lucasene and $L$-fence.
The basic properties of matchable Lucas cubes are obtained.
In addition, the rank generating functions, the cube polynomials, the maximal cube polynomials, the disjoint cube polynomials, the degree spectrum polynomials and the indegree (or outdegree) spectrum polynomials are considered.
And the relation between rank generating functions and Chebyshev polynomials, and the relation between (maximal) cube (or indegree sequence) polynomials and Lucas triangle are found.

\section{Preliminaries}

A set $P$ equipped with a binary relation $\le$ satisfying reflexivity, antisymmetry and transitivity is said to be a \emph{partially ordered set} (poset for short).
Given any poset $P$, the \emph{dual} $P^*$ of $P$ is formed by defining $x\le y$ to hold in $P^*$ if and only if $y\le x$ holds in $P$.
A subposet $S$ of $P$ is a \emph{chain} if any two elements of $S$ are comparable, and denoted by $\mathbf{n}$ if $|S|=n$ \cite{bDaveyP02}.
Let $x\prec y$ denote $y$ \emph{cover}s $x$ in $P$, i.e.\ $x<y$ and $x\le z < y$ implies $z=x$.
For $x \in P$, ${\uparrow x} = \{\, y \in P \mid x \le y \,\}$ and ${\downarrow x} = \{\, z \in P \mid z \le x \,\}$.
The subset $S$ of the poset $P$ is called \emph{convex} if $a,b \in S$, $c \in P$, and $a \le c \le b$ imply that $c \in S$.
Let $P$ be a poset and $F \subseteq P$. The subposet $F$ is a \emph{filter} if, whenever $x \in F$, $y \in P$ and $x \le y$, we have $y \in F$ \cite{bDaveyP02}.
The set of all filters of a poset $P$ is denoted by $\mathcal{F}(P)$, and carries the usual anti-inclusion order, forms a finite distributive lattice \cite{bDaveyP02,bStanl11} called \emph{filter lattice}.
For a finite lattice $L$, we denote by $\hat0_L$ (resp. $\hat1_L$) the minimum (resp. maximum) element in $L$.

The symmetric difference of two finite sets $A$ and $B$ is defined as $A\oplus B := (A\cup B)\setminus(A\cap B)$. If $M$ is a perfect matching of a graph and $C$ is an $M$-alternating cycle of the graph, then the symmetric difference of $M$ and edge-set $E(C)$ is another perfect matching of the graph, which is simply denoted by $M\oplus C$.
Let $G$ be a plane bipartite graph with a perfect matching, and the vertices of $G$ are colored properly black and white such that the two ends of every edge receive different colors.
An $M$-alternating cycle of $G$ is said to be \emph{proper}, if every edge of the cycle belonging to $M$ goes from white end-vertex to black end-vertex by the clockwise orientation of the cycle; otherwise \emph{improper} \cite{aZhangZ97b}.
An inner face of a graph is called a \emph{cell} if its boundary is a cycle, and we will say that the cycle is a cell too.

For some concepts and notations not explained in the paper, refer to \cite{bDaveyP02,bGraet11,bStanl11} for poset and lattice, \cite{bBondyM08,bHarar69} for graph theory.

Zhang and Zhang \cite{aZhangZ00} extended the concept of $Z$-transformation graph of a hexagonal system to plane bipartite graphs.

\begin{definition}[\cite{aZhangZ00}]
	Let $G$ be a plane bipartite graph. The $Z$-transformation graph $Z(G)$ is defined on $\mathcal{M}(G)$: $M_1, M_2 \in \mathcal{M}(G)$ are joined by an edge if and only if $M_1\oplus M_2$ is a cell of $G$.
	And  $Z$-transformation digraph $\vec{Z}(G)$ is the orientation of $Z(G)$: an edge $M_1M_2$ of $Z(G)$ is oriented from $M_1$ to $M_2$ if $M_1\oplus M_2$ form a proper $M_1$-alternating (thus improper $M_2$-alternating) cell.
\end{definition}

Let $G$ be a bipartite graph, from Theorem~4.1.1 in \cite{bLovasP86}, we have that $G$ is elementary if and only if $G$ is connected and every edge of $G$ lies in a perfect matching of $G$.
Let $G$ be a plane bipartite graph with a perfect matching, a binary relation $\le$ on $\mathcal{M}(G)$ is defined as: for $M_1,M_2\in\mathcal{M}(G)$, $M_1\le M_2$ if and only if $\vec{Z}(G)$ has a directed path from $M_2$ to $M_1$ \cite{aZhangZ00}.
In addition, Lam and Zhang \cite{aLamZ03} established the relationship between finite distributive lattices and $Z$-transformation directed graphs.

\begin{theorem}[\cite{aLamZ03}]\label{th:mfdl}
	If $G$ is a plane elementary bipartite graph, then $\mathcal{M}(G) := (\mathcal{M}(G),\le)$ is a finite distributive lattice and its Hasse diagram is isomorphic to $\vec{Z}(G)$.
\end{theorem}

Recently, Zhang et al.\ \cite{aZhangYY14} introduced the concept of matchable distributive lattice by Theorem~\ref{th:mfdl}.
\begin{definition}[\cite{aZhangYY14}]
	A finite distributive lattice $L$ is matchable if there is a plane weakly elementary bipartite graph $G$ such that $L\cong\mathcal{M}(G)$; otherwise it is non-matchable.
\end{definition}

The Lucas numbers is defined as follows: $L_0=2$, $L_1=1$ and $L_n=L_{n-1} + L_{n-2}$ for $n\ge2$. The generating function of $L_n$ is
\[
\sum_{n=0}^{\infty} L_nx^n = \frac{2-x}{1-x-x^2}.
\]

The Lucas triangle \cite{bKoshy01} $Y$ (or see A029635 in \cite{Sloan19}) is shown in Table~\ref{tab:jy}, and the entry in the $n$-th row and $k$-th column of Lucas triangle is given by
\[
Y(n,k) = \binom nk+\binom{n-1}{k-1} = Y(n-1,k-1) + Y(n-1,k),
\]
where $0 \le k \le n$ and $\binom{-1}{-1} = 1$.

\begin{table}[!htbp]
	\caption{The first six rows of Lucas triangle $Y$}\label{tab:jy}
	\[
	\begin{matrix}
	2 \\
	1 & 2 \\
	1 & 3 & 2 \\
	1 & 4 & 5 & 2 \\
	1 & 5 & 9 & 7 & 2 \\
	1 & 6 & 14 & 16 & 9 & 2 \\
	\end{matrix}
	\]
\end{table}

It is similar to Fibonacci numbers with binomial coefficients that Lucas numbers are given by Lucas triangle \cite{bKoshy01}. That is for $n \ge 0$,
\[
\sum_{k \ge 0} Y(n-k,k) = L_n.
\]

Let $G$ be a $2$-connected outerplanar bipartite graph. Let $P(G)$ denote the poset of $G$ determined from the directed inner dual graph of $G$ \cite{aZhangYY14}.
Moreover, Zhang et al.\ \cite{aZhangYY14} proved the follow Theorem~\ref{th:mfpg}.
\begin{theorem}[\cite{aZhangYY14}]\label{th:mfpg}
	Let $G$ be a $2$-connected outerplanar bipartite graph, and let $P(G)$ be the poset of $G$,
	\[
	\mathcal{M}(G) \cong \mathcal{F}(P(G)).
	\]
\end{theorem}

Let $L$ be a lattice and $I$ is a interval of $L$, Day \cite{aDay70} introduced a double structure $L[I] := (L \setminus I) \cup (I \mathbin{\square} \mathbf{2})$ and defined $x \le y$ in $L[I]$ if and only if one of the following hold:
\begin{enumerate}[(1)]
	\item $x,y \in L \setminus I$ and $x \le y \text{ in } L$;
	\item $x=(a,i)$, $y \in L \setminus I$ and $a \le y \text{ in } L$;
	\item $x \in L \setminus I$, $y=(b,j)$ and $x \le b \text{ in } L$;
	\item $x=(a,i)$, $y=(b,j)$ and $a \le b \text{ in } L$ and $i \le j \text{ in } \mathbf{2}$.
\end{enumerate}
We denote the distributive lattice $L[K]$ by $L \boxplus K$ \cite{aWangZY18} if $L$ is a finite distributive lattice and interval $K$ is a cutting (sublattice) \cite{aWangZY18} of $L$, i.e.\ if $L = {\downarrow \hat{1}_K} \cup {\uparrow \hat{0}_K}$ \cite{aDay92,aDayGP79}.
Let $P$ be a poset and $x \in P$. Let $P-x := P \setminus \{x\}$ and $P*x := P \setminus ({\uparrow x} \cup {\downarrow x})$. Wang et al.\ obtained a decomposition for filter lattice.

\begin{theorem}[\cite{aWangZY18}]\label{th:cetfdl}
	If $P$ is a poset and $x \in P$, then
	\[
	\mathcal{F}(P) = \mathcal{F}(P-x) \boxplus \mathcal{F}(P*x).
	\]
\end{theorem}

Let $[x^n]g(x)$ denote the coefficient of $x^n$ in the power series expansion of $g(x)$ \cite{bWilf94}.
A perfectly obvious property of this symbol, which we will use repeatedly, is $[x^n]\{x^mg(x)\} = [x^{n-m}]g(x)$.

\section{Matchable Lucas cubes}
\subsection{Lucasenes}

\begin{definition}
	A \emph{lucasene} is a hexagonal chain in which no three hexagons are linearly attached other than exactly three hexagons are linearly attached at one end.
\end{definition}

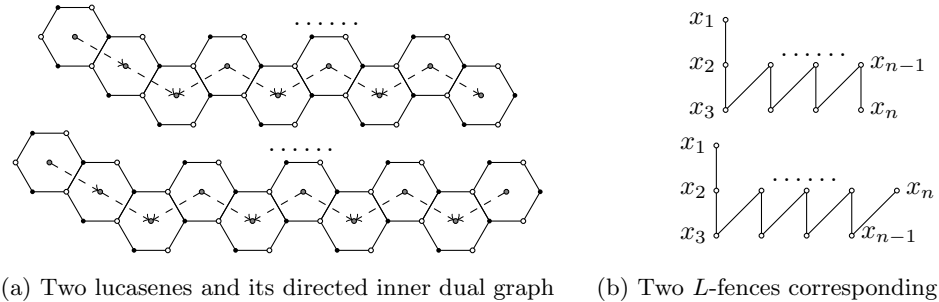
\begin{figure}[!ht]
	\centering
	\begin{subfigure}[b]{.5\linewidth}
		\centering
		\begin{tikzpicture}[scale=0.45]
		\newcommand{\ke}{3}; 
		
		\draw (1-1.5,{sqrt(3)/2}) -- (0.5-1.5,{sqrt(3)}) -- (-0.5-1.5,{sqrt(3)}) -- (-1-1.5,{sqrt(3)/2}) -- (-0.5-1.5,0) -- (0.5-1.5,0) -- cycle;
		\foreach \i in {0,...,\ke}
		{
			\foreach \j in {0,1}
			{
				\draw (1+1.5*\j+3*\i,-{0.5*\j*sqrt(3)}) -- (0.5+1.5*\j+3*\i,{sqrt(3)/2-0.5*\j*sqrt(3)}) -- (-0.5+1.5*\j+3*\i,{sqrt(3)/2-0.5*\j*sqrt(3)}) -- (-1+1.5*\j+3*\i,-{0.5*\j*sqrt(3)}) -- (-0.5+1.5*\j+3*\i,-{sqrt(3)/2-0.5*\j*sqrt(3)}) -- (0.5+1.5*\j+3*\i,-{sqrt(3)/2-0.5*\j*sqrt(3)}) -- cycle;
			}
		}
		
		\foreach \i in {0,...,\ke}
		{
			\foreach \j in {-1,1}
			{
				\foreach \k in {0,1}
				{
					\fill (3*\i-0.5+1.5*\k,{(\j-\k)*sqrt(3)/2}) circle (2pt);
					\filldraw[fill=white] (3*\i-0.5+1.5*\k+1,{(\j-\k)*sqrt(3)/2}) circle (2pt);
				}
			}
		}

		\foreach \j in {-1,1}
		{
			\fill (-1.5-0.5,{(1+\j)*sqrt(3)/2}) circle (2pt);
			\filldraw[fill=white] (-1.5-0.5+1,{(1+\j)*sqrt(3)/2}) circle (2pt);
		}
		\filldraw[fill=white] (-1.5-0.5-0.5,{sqrt(3)/2}) circle (2pt);
		\fill (3*\ke+1.5+1,-{sqrt(3)/2}) circle (2pt);
		
		\foreach \i in {1,...,\ke}
		{
			\foreach \j in {-1,1}
			{
				\draw[dashed,->] (\i*3+0.05*\j,-0.029) -- (\j*1.5+\i*3-0.05*\j,-0.866+0.029);
			}
		}
		\draw[dashed,->] (-2*1.5+1.5+0.05,0.866-0.029) -- (-0.05,0.029);
		\draw[dashed,->] (-0.05,0.029) -- (-0.05+1.5,0.029-0.866);
		
		\foreach \i in {0,...,\ke}
		{
			\foreach \j in {0,-1}
			{
				\filldraw[fill=gray] (\j*1.5+\i*3+1.5,-\j*0.866-0.866) circle (2pt);
			}
		}
		\filldraw[fill=gray] (-2*1.5+1.5,0.866) circle (2pt);
		
		\node at (6,1.25) {\dots \dots};
		\end{tikzpicture}
		
		\begin{tikzpicture}[scale=0.45]
		\newcommand{\ke}{3}; 
		
		\draw (1-1.5,{sqrt(3)/2}) -- (0.5-1.5,{sqrt(3)}) -- (-0.5-1.5,{sqrt(3)}) -- (-1-1.5,{sqrt(3)/2}) -- (-0.5-1.5,0) -- (0.5-1.5,0) -- cycle;
		\foreach \i in {0,...,\ke}
		{
			\foreach \j in {0,1}
			{
				\draw (1+1.5*\j+3*\i,-{0.5*\j*sqrt(3)}) -- (0.5+1.5*\j+3*\i,{sqrt(3)/2-0.5*\j*sqrt(3)}) -- (-0.5+1.5*\j+3*\i,{sqrt(3)/2-0.5*\j*sqrt(3)}) -- (-1+1.5*\j+3*\i,-{0.5*\j*sqrt(3)}) -- (-0.5+1.5*\j+3*\i,-{sqrt(3)/2-0.5*\j*sqrt(3)}) -- (0.5+1.5*\j+3*\i,-{sqrt(3)/2-0.5*\j*sqrt(3)}) -- cycle;
			}
		}
		
		\draw (1+3*\ke+3,0) -- (0.5+3*\ke+3,{sqrt(3)/2}) -- (-0.5+3*\ke+3,{sqrt(3)/2}) -- (-1+3*\ke+3,0) -- (-0.5+3*\ke+3,-{sqrt(3)/2}) -- (0.5+3*\ke+3,-{sqrt(3)/2}) -- cycle;
		
		\foreach \i in {0,...,\ke}
		{
			\foreach \j in {-1,1}
			{
				\foreach \k in {0,1}
				{
					\fill (3*\i-0.5+1.5*\k,{(\j-\k)*sqrt(3)/2}) circle (2pt);
					\filldraw[fill=white] (3*\i-0.5+1.5*\k+1,{(\j-\k)*sqrt(3)/2}) circle (2pt);
				}
			}
		}

		\foreach \j in {-1,1}
		{
			\fill (3*\ke+3-0.5,{\j*sqrt(3)/2}) circle (2pt);
			\filldraw[fill=white] (3*\ke+3-0.5+1,{\j*sqrt(3)/2}) circle (2pt);
		}
		
		\foreach \j in {-1,1}
		{
			\fill (-1.5-0.5,{(1+\j)*sqrt(3)/2}) circle (2pt);
			\filldraw[fill=white] (-1.5-0.5+1,{(1+\j)*sqrt(3)/2}) circle (2pt);
		}
		\filldraw[fill=white] (-1.5-0.5-0.5,{sqrt(3)/2}) circle (2pt);
		\fill (3*\ke+3+1,0) circle (2pt);
		
		\foreach \i in {0,...,\ke}
		{
			\foreach \j in {-1,1}
			{
				\draw[dashed,->] (\j*1.5+\i*3+1.5-0.05*\j,0-0.029) -- (\i*3+1.5+0.05*\j,-0.866+0.029);
			}
		}
		\draw[dashed,->] (-2*1.5+1.5+0.05,0.866-0.029) -- (-0.05,0.029);
		
		\foreach \i in {0,...,\ke}
		{
			\foreach \j in {0,1}
			{
				\filldraw[fill=gray] (\j*1.5+\i*3+1.5,\j*0.866-0.866) circle (2pt);
			}
		}
		\filldraw[fill=gray] (-2*1.5+1.5,0.866) circle (2pt)
		(0,0) circle (2pt);
		
		\node at (6,1.25) {\dots \dots};
		\end{tikzpicture}
		\caption{Two lucasenes and its directed inner dual graph}\label{fig:lucasenes}
	\end{subfigure}%
	\begin{subfigure}[b]{.35\linewidth}
		\centering
		\begin{tikzpicture}[scale=0.6]
		\newcommand\ke{3};
		
		\foreach \i in {1,...,\ke}
		{
			\draw (\i-1,0) -- (\i,1) -- (\i,0);
		}
		\draw (0,0) -- (0,2);
		
		\foreach \i in {0,...,\ke}
		{
			\foreach \j in {0,1}
			{
				\filldraw[fill=white] (\i,\j) circle (1.5pt);
			}
		}
		\filldraw[fill=white] (0,2) circle (1.5pt);
		
		\foreach \j in {1,2,3}
		{
			\node[left] at (0,3-\j) {$x_{\j}$};
		}
		\node[right] at (\ke,1) {$x_{n-1}$};
		\node[right] at (\ke,0) {$x_n$};
		
		\node[above] at (\ke-1,1) {\dots \dots};
		\end{tikzpicture}
		
		\begin{tikzpicture}[scale=0.6]
		\newcommand\ke{3};
		
		\foreach \i in {1,...,\ke}
		{
			\draw (\i-1,0) -- (\i,1) -- (\i,0);
		}
		\draw (0,0) -- (0,2);
		\draw (\ke,0) -- (\ke+1,1);
		
		\foreach \i in {0,...,\ke}
		{
			\foreach \j in {0,1}
			{
				\filldraw[fill=white] (\i,\j) circle (1.5pt);
			}
		}
		\filldraw[fill=white] (0,2) circle (1.5pt);
		\filldraw[fill=white] (\ke+1,1) circle (1.5pt);
		
		\foreach \j in {1,2,3}
		{
			\node[left] at (0,3-\j) {$x_{\j}$};
		}
		\node[right] at (\ke,0) {$x_{n-1}$};
		\node[right] at (\ke+1,1) {$x_n$};
		
		\node[above] at (\ke-1,1) {\dots \dots};
		\end{tikzpicture}
		\caption{Two $L$-fences corresponding to \ref{fig:lucasenes}}\label{fig:lfences}
	\end{subfigure}
	\caption{Two lucasenes with directed inner dual graph and two $L$-fences corresponding to them}\label{fig:lucasenes-lfences}
\end{figure}

\begin{definition}
	An \emph{$L$-fence} $\Xi_n$ is a poset the Hasse diagram of which is isomorphic to the directed inner dual graph of lucasene with $n$ hexagons.
\end{definition}

By Theorems~\ref{th:mfdl} and \ref{th:mfpg}, Theorem~\ref{th:ztgfl} is obvious.
\begin{theorem}\label{th:ztgfl}
  The $Z$-transformation directed graph of lucasene with $n$ hexagons is isomorphic to the Hasse diagram of the filter lattice of $\Xi_n$.
\end{theorem}

It is similar to \cite[Exercise~1.35(e)]{bStanl11} that $|\mathcal{F}(\Xi_n)| = L_n$ for $n\ge 2$, therefore the matchable Lucas distributive lattices and the matchable Lucas cubes are introduced.

\begin{definition}
	The filter lattice $\mathcal{F}(\Xi_n)$ of $L$-fence $\Xi_n$ is called the $n$-th \emph{matchable Lucas distributive lattice}, denoted by $\Omega_n$; and its undirected Hasse diagram is called the $n$-th \emph{matchable Lucas cube}, denoted by $\Omega_n$ too.
\end{definition}

For convenience, the (directed) Hasse diagram of $\mathcal{F}(\Xi_n)$ is denoted by $\Omega_n$ too, and let $|\Omega_0| = 1$.
The first eight matchable Lucas cubes are shown in Figure~\ref{fig:mlc}.

\begin{figure}[!htb]
  \centering
  \begin{tikzpicture}[scale=0.5]
	  \filldraw[fill=white] (0,0) circle (1.5pt);
  \end{tikzpicture}\quad
  \begin{tikzpicture}[scale=0.5]
	  \draw (0,0) -- (0,1);
	  \filldraw[fill=white] (0,0) circle (1.5pt);
	  \filldraw[fill=white] (0,1) circle (1.5pt);
  \end{tikzpicture}\quad
  \begin{tikzpicture}[scale=0.5]
	  \draw (0,0) -- (0,2);
	  
	  \foreach \j in {0,1,2}
	  {
	  	\filldraw[fill=white] (0,\j) circle (1.5pt);
	  }
  \end{tikzpicture}\quad
  \begin{tikzpicture}[scale=0.5]
	  \draw (0,0) -- (0,3);
	  
	  \foreach \j in {0,...,3}
	  {
	  	\filldraw[fill=white] (0,\j) circle (1.5pt);
	  }
  \end{tikzpicture}\quad
  \begin{tikzpicture}[scale=0.5]
    \foreach \i in {0,1,2}
    {
    	\draw (\i,\i) -- (\i-1,\i+1);
    }
    
    \draw (0,-1) -- (0,0) -- (2,2) (-1,1) -- (1,3);
    
    \foreach \i in {0,1,2}
    {
    	\foreach \j in {0,1}
    	{
    		\filldraw[fill=white] (\i-\j,\i+\j) circle (1.5pt);
    	}
    }
    \filldraw[fill=white] (0,-1) circle (1.5pt);
  \end{tikzpicture}\quad
  \begin{tikzpicture}[scale=0.5]
    \foreach \j in {0,1}
    {
    	\foreach \i in {0,1,2}
    	{
    		\draw (\i-\j+1,\i+\j+1) -- (\i-\j,\i+\j+2)
    		(\i-\j,\i+\j) -- (\i-\j+1,\i+\j+1);
    	}
    	\draw (-1+\j,3+\j) -- (-1+\j+1,3+\j+1);
    }
    
    \draw (0,0) -- (-1,1);
    
    \foreach \i in {0,1,2}
    {
    	\foreach \j in {0,1,2}
    	{
    		\filldraw[fill=white] (\i-\j+1,\i+\j+1) circle (1.5pt);
    	}
    }
    \filldraw[fill=white] (0,0) circle (1.5pt);
    \filldraw[fill=white] (-1,1) circle (1.5pt);
  \end{tikzpicture}\quad
  \begin{tikzpicture}[scale=0.5]
    \foreach \i in {0,1,2,3}
    {
    	\foreach \j in {0,1}
    	{
    		\draw (\i-\j,\i+\j) -- (\i-\j-1,\i+\j+1);
    	}
    }
    
    \foreach \i in {0,1,2}
    {
    	\foreach \j in {0,1,2}
    	{
    		\draw (\i-\j,\i+\j) -- (\i-\j+1,\i+\j+1);
    	}
    }
    
    \foreach \i in {0,1,2}
    {
    	\draw (\i,\i+3) -- (\i-1,\i+1+3);
    }
    
    \foreach \i in {0,1}
    {
    	\foreach \j in {0,1}
    	{
    		\draw (\i-\j,\i+\j+3) -- (\i-\j+1,\i+\j+1+3);
    	}
    }
    
    \foreach \i in {0,1,2}
    {
    	\foreach \j in {0,1}
    	{
    		\draw (\i-\j,\i+\j+2) -- (\i-\j,\i+\j+3);
    	}
    }
    
    \foreach \i in {0,1,2}
    {
    	\foreach \j in {0,1}
    	{
    		\filldraw[fill=white] (\i-\j,\i+\j+3) circle (1.5pt);
    	}
    }
    
    \foreach \i in {0,1,2,3}
    {
    	\foreach \j in {0,1,2}
    	{
    		\filldraw[fill=white] (\i-\j,\i+\j) circle (1.5pt);
    	}
    }
  \end{tikzpicture}\quad
  \begin{tikzpicture}[scale=0.5]
    \foreach \i in {1,2,3}
    {
    	\foreach \k in {0,1}
    	{
    		\foreach \j in {0,1}
    		{
    			\draw (\i-2*\k,\i+\k+\j) -- (\i-2*\k,\i+\k+\j+1)
    			(-2*\k+1+\j,\i+\k+\j) -- (-2*\k+1+\j+1,\i+\k+\j+1);
    		}
    	}
    }
    
    \foreach \i in {1,2,3}
    {
    	\foreach \j in {1,2,3}
    	{
    		\draw (\i,\i+\j-1) -- (\i-2,\i+\j);
    	}
    }
    
    \foreach \i in {0,1}
    {
    	\foreach \j in {0,1}
    	{
    		\draw 
    		(-2*\j,\i+\j) -- (-2*\j+1,\i+\j+1);
    	}
    	\draw (0,\i) -- (-2,\i+1)
    	(-2*\i,\i) -- (-2*\i,\i+1);
    }
    
    \foreach \i in {-1,0,1}
    {
    	\foreach \j in {4,5}
    	{
    		\draw (\i,\i+\j) -- (\i-2,\i+\j+1);
    	}
    	\draw (\i-2,\i+4+1) -- (\i-2,\i+4+2);
    }
    
    \foreach \i in {0,1}
    {
    	\foreach \j in {4,5}
    	{
    		\draw (-3+\i,\i+\j) -- (\i-2,\i+\j+1);
    	}
    }
    
    \draw (-2,2) -- (-4,3) -- (-3,4);
    
    \foreach \i in {1,2,3}
    {
    	\foreach \j in {1,2,3}
    	{
    		\foreach \k in {0,1}
    		{
    			\filldraw[fill=white] (\i-2*\k,\i+\j+\k-1) circle (1.5pt);
    		}
    	}
    }
    
    \foreach \i in {-2,-1,0}
    {
    	\foreach \j in {4,5}
    	{
    		\filldraw[fill=white] (\i-1,\i+\j+2) circle (1.5pt);
    	}
    }
    
    \foreach \i in {0,1}
    {
    	\foreach \j in {0,1}
    	{
    		\filldraw[fill=white] (-2*\j,\i+\j) circle (1.5pt);
    	}
    }
    
    \filldraw[fill=white] (-4,3) circle (1.5pt);
  \end{tikzpicture}
  \caption{The first eight matchable Lucas cubes $\Omega_0$, $\Omega_1$, \dots, $\Omega_7$}\label{fig:mlc}
\end{figure}

The structures of matchable Lucas cubes can be given as follows, as shown in Figures~\ref{fig:struc} and \ref{fig:strucf}.
\begin{theorem}\label{th:struc}
  Let $\Omega_n$ be the $n$-th matchable Lucas cube. For $n \ge 4$,
  \[
  \Omega_n \cong \Omega_{n-1} \boxplus \Omega_{n-2} \cong (\Omega_{n-2} \boxplus \Omega_{n-2}) \boxplus \Omega_{n-3};
  \]
  or for $n \ge 3$,
  \[
  \Omega_n \cong \Gamma_{n-1} \boxplus \Gamma_{n-3} \cong (\Gamma_{n-2}^* \boxplus \Gamma_{n-3}) \boxplus \Gamma_{n-3}.
  \]
\end{theorem}
\begin{proof}
	By Theorem~\ref{th:cetfdl}, for poset $\Xi_n$ as shown in Figure~\ref{fig:lfences}, we have for $n \ge 4$,
	\begin{align*}
	\mathcal{F}(\Xi_n) &= \mathcal{F}(\Xi_n-x_n) \boxplus \mathcal{F}(\Xi_n*x_n) = \mathcal{F}(\Xi_{n-1}) \boxplus \mathcal{F}(\Xi_{n-2}) \\
	&= (\mathcal{F}((\Xi_n-x_n)-x_{n-1}) \boxplus \mathcal{F}((\Xi_n-x_n)*x_{n-1})) \boxplus \mathcal{F}(\Xi_n*x_n) = (\mathcal{F}(\Xi_{n-2}) \boxplus \mathcal{F}(\Xi_{n-2})) \boxplus \mathcal{F}(\Xi_{n-3});
	\end{align*}
	by Theorem~\ref{th:mfpg}, combining $Z_{n-1} = \Xi_n - x_1 = \Xi_n - x_2$ \cite{aMunarZ02b,bStanl11}, the directed inner dual graph of fibonaccene with $n$ hexagons and $Z_n$ are isomorphic \cite{aYaoZ15}, and Fibonacci cubes are the resonance graphs of fibonaccences \cite{aKlavzZ05}, we also have for $n \ge 3$,
	\begin{align*}
	\mathcal{F}(\Xi_n) &= \mathcal{F}(\Xi_n - x_1) \boxplus \mathcal{F}(\Xi_n * x_1) = \mathcal{F}(Z_{n-1}) \boxplus \mathcal{F}(Z_{n-3}) \\
	&= ((\mathcal{F}((\Xi_n - x_1) - x_2) \boxplus \mathcal{F}((\Xi_n - x_1) * x_2))) \boxplus \mathcal{F}(\Xi_n * x_1) = (\mathcal{F}(Z_{n-2})^* \boxplus \mathcal{F}(Z_{n-3})) \boxplus \mathcal{F}(Z_{n-3}).
	\end{align*}
	
	Therefore the structures of matchable Lucas cubes are obtained.
\end{proof}

\begin{figure}[!htb]
	\centering
	\begin{subfigure}[b]{.4\linewidth}
		\centering
		\begin{tikzpicture}[scale=0.6,rotate=-30]
			\draw (-3,2) -- (0,2) (-3,-2) -- (2,-2) (-3,0.448) -- (2,0.448);
			\draw (-3,0) ellipse (1cm and 2cm)
			(0,0) ellipse (1cm and 2cm)
			(2,-0.776) ellipse (0.612cm and 1.224cm)
			(0,-0.776) ellipse (0.612cm and 1.224cm)
			(-3,-0.776) ellipse (0.612cm and 1.224cm);
			\draw[dashed] (0.95,0) circle (2.35cm);
			
			\node at (1.5,1.8) {$\Omega_{n-1}$};
			\node at (0,1) {$\Omega_{n-2}$};
			\node at (-3,1) {$\Omega_{n-2}$};
			\node at (0,-0.776) {$\Omega_{n-3}$};
			\node at (-3,-0.776) {$\Omega_{n-3}$};
			\node at (2,-0.776) {$\Omega_{n-3}$};
		\end{tikzpicture}
		\caption{$2 \mid n$}\label{fig:estruc}
	\end{subfigure}
	\begin{subfigure}[b]{.4\linewidth}
		\begin{tikzpicture}[scale=0.6,rotate=-30]
			\draw (3,-2) -- (0,-2) (3,2) -- (-2,2) (3,-0.448) -- (-2,-0.448);
			\draw (3,0) ellipse (1cm and 2cm)
			(0,0) ellipse (1cm and 2cm)
			(-2,0.776) ellipse (0.612cm and 1.224cm)
			(0,0.776) ellipse (0.612cm and 1.224cm)
			(3,0.776) ellipse (0.612cm and 1.224cm);
			\draw[dashed] (-0.95,0) circle (2.35cm);
			
			\node at (-1.5,-1.8) {$\Omega_{n-1}$};
			\node at (0,-1) {$\Omega_{n-2}$};
			\node at (3,-1) {$\Omega_{n-2}$};
			\node at (0,0.776) {$\Omega_{n-3}$};
			\node at (3,0.776) {$\Omega_{n-3}$};
			\node at (-2,0.776) {$\Omega_{n-3}$};
		\end{tikzpicture}
		\caption{$2 \nmid n$}\label{fig:ostruc}
	\end{subfigure}
	\caption{The structure of $\Omega_n$ for $n \ge 4$}\label{fig:struc}
\end{figure}
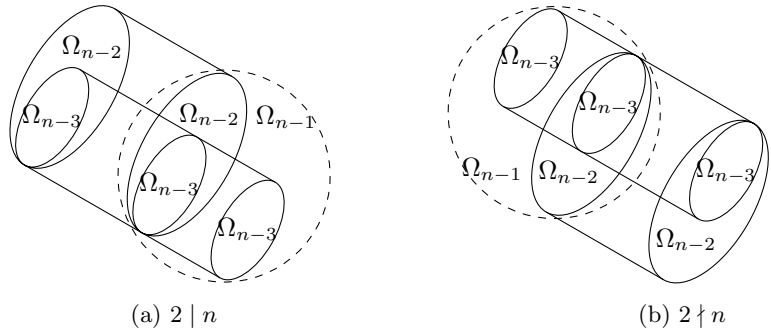

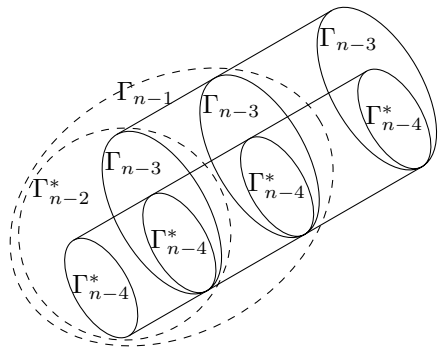
\begin{figure}[!ht]
	\centering
	\begin{tikzpicture}[scale=0.6,rotate=30]
			\draw (-2,-2) -- (5.5,-2) (0,2) -- (5.5,2) (-2,0.448) -- (5.5,0.448);
		\draw (5.5,0) ellipse (1cm and 2cm)
		(2.5,0) ellipse (1cm and 2cm)
		(0,0) ellipse (1cm and 2cm)
		(-2,-0.776) ellipse (0.612cm and 1.224cm)
		(0,-0.776) ellipse (0.612cm and 1.224cm)
		(2.5,-0.776) ellipse (0.612cm and 1.224cm)
		(5.5,-0.776) ellipse (0.612cm and 1.224cm);;
		\draw[dashed] (-0.95,0) circle (2.35cm)
		(0.25,0) ellipse (3.8cm and 2.8cm);
		
		\node at (1,2.45) {$\Gamma_{n-1}$};
		\node at (-1.7,1.5) {$\Gamma_{n-2}^*$};
		\node at (5.5,1.25) {$\Gamma_{n-3}$};
		\node at (0,1.25) {$\Gamma_{n-3}$};
		\node at (2.5,1.25) {$\Gamma_{n-3}$};
		\node at (-2,-0.776) {$\Gamma_{n-4}^*$};
		\node at (0,-0.776) {$\Gamma_{n-4}^*$};
		\node at (2.5,-0.776) {$\Gamma_{n-4}^*$};
		\node at (5.5,-0.776) {$\Gamma_{n-4}^*$};
	\end{tikzpicture}
	\caption{The structure of $\Omega_n$ given by $\Gamma_n$}\label{fig:strucf}
\end{figure}

\begin{proposition}
  The both height and diameter of $\Omega_n$ are $n+1$ \cite{aZhangLS08}, and thus the radius is $\lfloor \frac{n+2}2 \rfloor$.
  In addition, $\Omega_n$ is non-Eulerian and has a Hamiltonian path \cite{aZhangZY04}.
\end{proposition}

\subsection{Rank generating functions}\label{ssec:rank}

Let $r_{n,k} := r_k(\Omega_n)$ denote the number of elements of rank $k$ in $\Omega_n$, and let $R_n(x) := R(\Omega_n,x)= \sum_{k\ge 0}r_{n,k}x^k$ be the rank generating function of $\Omega_n$ \cite[P291]{bStanl11}.
The first few of $R_n(x)$ is listed.
\begin{align*}
R_0(x) &= 1 \\
R_1(x) &= 1+x \\
R_2(x) &= 1+x+x^2 \\
R_3(x) &= 1+x+x^2+x^3 \\
R_4(x) &= 1+x+2x^2+2x^3+x^4 \\
R_5(x) &= 1+2x+2x^2+3x^3+2x^4+x^5 
\end{align*}

It is obvious that $R_n(x)$ is always a polynomial with degree $n$. By Corollary~4.2 in \cite{aWangZY18} and Theorem~\ref{th:struc}, we have Propositions~\ref{prop:rec-Rf} and \ref{prop:rec-R}.
\begin{proposition}\label{prop:rec-Rf}
	For $n \ge 3$,
	\[
	R_n(x) = R(\Gamma_{n-1},x) + x^3R(\Gamma_{n-3},x),
	\]
	where $R(\Gamma_n,x)$ is the rank generating function of $\Gamma_n$ \cite{aMunarZ02b,bStanl11}.
\end{proposition}

\begin{proposition}\label{prop:rec-R}
	For $n\ge 4$
	\[
	R_n(x) =
	\begin{cases}
	xR_{n-1}(x) + R_{n-2}(x), & \mbox{if } 2\nmid n, \\
	R_{n-1}(x) + x^2R_{n-2}(x), & \mbox{if } 2\mid n.
	\end{cases}
	\]
\end{proposition}

In other words, for $m\ge 2$,
\begin{align*}
\begin{cases}
R_{2m}(x) = R_{2m-1}(x) + x^2R_{2m-2}(x), \\
R_{2m+1}(x) = xR_{2m}(x) + R_{2m-1}(x).
\end{cases}
\end{align*}
Let $A_m(x) = R_{2m}(x)$ and let $B_m(x) = R_{2m+1}(x)$. For $m \ge 2$, we have
\begin{align*}
\begin{cases}
A_m(x) = B_{m-1}(x) + x^2A_{m-1}(x), \\
B_m(x) = xA_m(x) + B_{m-1}(x).
\end{cases}
\end{align*}

Note that $R_2(1)=3$ and $R_3(1)=4$, by Proposition~\ref{prop:rec-R}, thus for $n \ge 2$,
$R_n(1) = L_n$;
for $m \ge 1$, $A_m(1) = L_{2m}$ and $B_m(1) = L_{2m+1}$.
In addition, we have the recurrence relations of $A_m(x)$ and $B_m(x)$.
\begin{proposition}\label{prop:rec-AB}
	Let $A_m(x) = R_{2m}(x)$ and let $B_m(x) = R_{2m+1}(x)$. For $m \ge 3$,
	\[
	\begin{cases}
	A_m(x) = (1+x+x^2)A_{m-1}(x) - x^2A_{m-2}(x), \\
	B_m(x) = (1+x+x^2)B_{m-1}(x) - x^2B_{m-2}(x).
	\end{cases}
	\]
\end{proposition}
\begin{proof}
	By Proposition~\ref{prop:rec-R}, for $m\ge 3$,
	\begin{align*}
	A_m(x) &= B_{m-1}(x) + x^2A_{m-1}(x) \\
	&= xA_{m-1}(x) + B_{m-2}(x) + x^2A_{m-1}(x) \\
	&= (x+x^2)A_{m-1}(x) + A_{m-1}(x) - x^2A_{m-2}(x) \\
	&= (1+x+x^2)A_{m-1}(x) - x^2A_{m-2}(x).
	\end{align*}
	
	Likewise, the recurrence relation of $B_m(x)$ can be obtained.
\end{proof}

Considering $A_m(x)$ and $B_m(x)$ with $x=1$ in Proposition~\ref{prop:rec-AB}, we see a trivial result on $L_n$:
for $n \ge 4$, $L_n = 3L_{n-2} - L_{n-4}$.
And we get the generating functions of $A_m(x)$ and $B_m(x)$ by Proposition~\ref{prop:rec-AB}, respectively.
\begin{theorem}\label{th:gf-AB}
	The generating functions of $A_m(x)$ and $B_m(x)$ are
	\[
	\sum_{m\ge 0}A_m(x)z^m = \frac{1-xz^2}{1-(1+x+x^2)z+x^2z^2}
	\]
	and
	\[
	\sum_{m\ge0}B_m(x)z^m = \frac{1+x^3z}{1-(1+x+x^2)z+x^2z^2}+x,
	\]
	respectively.
\end{theorem}
\begin{proof}
	By Proposition~\ref{prop:rec-AB},
	\begin{align*}
	\sum_{m\ge 0}A_m(x)z^m &= \sum_{m\ge 3}A_m(x)z^m + A_2(x)z^2 + A_1(x)z+A_0(x) \\
	&= \sum_{m\ge 3}((1+x+x^2)A_{m-1}(x)-x^2A_{m-2}(x))z^m + A_2(x)z^2 + A_1(x)z+A_0(x) \\
	&= (1+x+x^2)z\sum_{m\ge 2} A_m(x)z^m - x^2z^2 \sum_{m\ge 1} A_m(x)z^m + A_2(x)z^2 + A_1(x)z+A_0(x) \\
	&= (1+x+x^2)z\sum_{m\ge 0} A_m(x)z^m - x^2z^2 \sum_{m\ge 0} A_m(x)z^m + 1 - xz^2
	\end{align*}

	The generating function of $B_m$ can be proved similarly.
\end{proof}

It should be pointed out that $A_m(x)$ and $B_m(x)$ could be given by Chebyshev polynomials.
Chebyshev polynomials $U_n(x)$ \cite{bMasonH02,bWilf94} of the second kind are defined as following: $U_0(x) = 1$, $U_1(x) = 2x$ and $U_n(x) = 2xU_{n-1}(x) - U_{n-2}(x)$ for $n \ge 2$.
And the generating function of $U_m(x)$ is 
\[
\sum_{n=0}^\infty U_n(x) y^n = \frac1{1-2xy+y^2}.
\]

\begin{theorem}
	For $m \ge 1$,
	\[
	A_m(x) = x^{m-2}(1+x)U_m\left(\frac{1+x+x^2}{2x}\right) - x^{m-2}(1 + x + x^2)U_{m-1}\left(\frac{1+x+x^2}{2x}\right),
	\]
	and
	\[
	B_m(x) = x^mU_m\left(\frac{1+x+x^2}{2x}\right) + x^{m+2}U_{m-1}\left(\frac{1+x+x^2}{2x}\right).
	\]
\end{theorem}
\begin{proof}
	Combining the generating function of $U_m(x)$ and Theorem~\ref{th:gf-AB},
	\begin{align*}
		A_m(x) &= [z^m] \frac{1-xz^2}{1-(1+x+x^2)z+x^2z^2} \\
		&= [z^m] \frac{1-xz^2}{1-2\frac{1+x+x^2}{2x}(xz)+(xz)^2} \\
		&= [z^m] \frac{1}{1-2\frac{1+x+x^2}{2x}(xz)+(xz)^2} - x[z^{m-2}] \frac{1}{1-2\frac{1+x+x^2}{2x}(xz)+(xz)^2} \\
		&= x^mU_m\left(\frac{1+x+x^2}{2x}\right) - x^{m-1}U_{m-2}\left(\frac{1+x+x^2}{2x}\right) \\
		&= x^{m-2}(1+x)U_m\left(\frac{1+x+x^2}{2x}\right) - x^{m-2}(1 + x + x^2)U_{m-1}\left(\frac{1+x+x^2}{2x}\right).
	\end{align*}
	
	Thus $B_m(x)$ is obtained by $B_m(x) = A_{m+1}(x) - x^2A_{m}(x)$ and a tedious calculations.
\end{proof}

Let $\binom{n;3}k$ \cite{bComte74} denote the coefficient of $x^k$ in $(1+x+x^2)^n$, namely
\[
\binom{n;3}k = \sum_{j=0}^{\lfloor k/2 \rfloor} \binom n{k-j} \binom{k-j}j,
\]
and see sequence A027907 in \cite{Sloan19}.
Using Kronecker delta function $\delta$, we have the coefficient $r_{n,k}$.
\begin{theorem}
	For $m \ge 0$,
	\[
	r_{2m,k} = \sum_{j=0}^{\lfloor m/2 \rfloor} \left(\binom{m-j}j \binom{m-2j;3}{k-2j} + \binom{m-j-1}{j-1} \binom{m-2j;3}{k-2j+1} \right) (-1)^j,
	\]
	and
	\[
	r_{2m+1,k} = \delta_{m0}\delta_{k1} + \sum_{j=0}^{\lfloor m/2 \rfloor} \left(\binom{m-j}j \binom{m-2j;3}{k-2j} + \binom{m-j-1}{j} \binom{m-2j-1;3}{k-2j-3} \right) (-1)^j.
	\]
\end{theorem}
\begin{proof}
	Consider the polynomials $g_n(x)$ defined by
	\[
	\sum_{n \ge 0} g_n(x) z^n = \frac1{1-(1+x+x^2)z+x^2z^2}
	\]
	such that
	\[
	g_n(x) = \sum_{j=0}^{\lfloor n/2 \rfloor} \binom{n-j}j x^{2j} (1+x+x^2)^{n-2j} (-1)^j.
	\]
	In addition, the coefficient of $x^k$ in $g_n(x)$ can be given by
	\[
	[x^k]g_n(x) = \sum_{j=0}^{\lfloor n/2 \rfloor} \binom{n-j}j [x^{k-2j}] (1+x+x^2)^{n-2j} (-1)^j = \sum_{j=0}^{\lfloor n/2 \rfloor} \binom{n-j}j \binom{n-2j;3}{k-2j} (-1)^j.
	\]
	Then, since $A_m(x) = g_m(x) - xg_m(x)$, we have
	\begin{align*}
		r_{2m,k} &= [x^k]A_m(x) = [x^k]g_m(x) - [x^{k-1}]g_{m-2}(x) \\
		&= \sum_{j=0}^{\lfloor m/2 \rfloor} \binom{m-j}j \binom{m-2j;3}{k-2j} (-1)^j - \sum_{j=0}^{\lfloor m/2-1 \rfloor} \binom{m-2-j}j \binom{m-2-2j;3}{k-2j-1} (-1)^j \\
		&= \sum_{j=0}^{\lfloor m/2 \rfloor} \binom{m-j}j \binom{m-2j;3}{k-2j} (-1)^j + \sum_{j=1}^{\lfloor m/2 \rfloor} \binom{m-j-1}{j-1} \binom{m-2j;3}{k-2j+1} (-1)^j \\
		&= \sum_{j=0}^{\lfloor m/2 \rfloor} \left(\binom{m-j}j \binom{m-2j;3}{k-2j} + \binom{m-j-1}{j-1} \binom{m-2j;3}{k-2j+1} \right) (-1)^j.
	\end{align*}
	Thus, $r_{2m+1,k}$ is obtained from $B_m(x) = x + g_m(x) + x^3g_{m-1}(x)$ in the same way.
\end{proof}

The generating function of $R_n(x)$ is also a straightforward consequence of Theorem~\ref{th:gf-AB}.
\begin{theorem}
	The generating function of $R_n(x)$ is
	\[
	\sum_{n\ge0}R_n(x)y^n = \frac{1+y+x^3y^3-xy^4}{1-(1+x+x^2)y^2+x^2y^4}+xy.
	\]
\end{theorem}
\begin{proof}
	By the definition of $A_m(x)$ and $B_m(x)$,
	\begin{align*}
	\sum_{n\ge0}R_n(x)y^n &= \sum_{m\ge 0}A_m(x)y^{2m} + \sum_{m\ge 0}B_m(x)y^{2m+1} \\
	&= \sum_{m\ge 0}A_m(x)y^{2m} + y\sum_{m\ge 0}B_m(x)y^{2m} \\
	&= \frac{1-xy^4}{1-(1+x+x^2)y^2+x^2y^4} + y\frac{1+x^3y^2}{1-(1+x+x^2)y^2+x^2y^4}+xy \\
	&= \frac{1+y+x^3y^3-xy^4}{1-(1+x+x^2)y^2+x^2y^4}+xy.
	\end{align*}
	
	The proof is completed.
\end{proof}

\textbf{Remark}: Since the thoughts and methods used in the following five subsections are exactly the same as those used in Subsection~\ref{ssec:rank}, we choose to list only the results and omit the proof process.

\subsection{Cube polynomials}

Let $q_{n,k} := q_k(\Omega_n)$ denote the number of $k$-dimensional induced hypercubes of $\Omega_{n}$, and let $q_{n,k} = 0$ if no $k$-dimensional induced hypercubes of $\Omega_{n}$ exists.
Since, for $n \ge 2$, the initial values and the recurrence relation are same as in \cite[Section~5]{aKlavM12} and \cite{aSaygE18} for $q=1$, the results of $q_{n,k}$ are same as $c_k(\Lambda_n)$ in \cite[Section~5]{aKlavM12} and $h_{n,k;1}$ in \cite{aSaygE18}.

\begin{proposition}\label{prop:dimq}
  The dimension of the maximum induced hypercube of $\Omega_n$ is $\lfloor \frac n2 \rfloor$.

  Moreover, the number is equal to the number of maximum in-degree (or out-degree) of $\Omega_n$ or the number of maximum anti-chain in $\Xi_n$.
\end{proposition}

\begin{lemma}[\cite{aWangZY18}]\label{lem:enum}
  Let $L$ be a finite distributive lattice. If $K$ is a cutting of $L$, then
  \[
  q_k(L\boxplus K) = q_k(L) + q_k(K) + q_{k-1}(K).
  \]
\end{lemma}

It is evident that the recurrence relation of $q_{n,k}$ follows from Lemma~\ref{lem:enum} .
\begin{proposition}\label{prop:rec-q}
  For $n\ge 4$,
  \[
  q_{n,k} = q_{n-1,k}+q_{n-2,k}+q_{n-2,k-1}.
  \]
\end{proposition}

By induction, it is not difficult to show that
the number of the maximum induced hypercubes of $\Omega_n$ is
\[
q_{n,\lfloor \frac n2 \rfloor} =
\begin{cases}
2, & \mbox{if } 2\mid n, \\
n, & \mbox{if } 2\nmid n.
\end{cases}
\]
Moreover, the number of vertices, edges, $4$-cycles and $3$-dimensional cubes of $\Omega_n$ is expressed.
\begin{theorem}\label{th:cenum4}
	The first four of $q_{n,k}$ is obtained.
	\begin{enumerate}[(1)]
		\item For $n \ge 2$,
			\[
			q_{n,0} = L_n,
			\]
			and 
			\[
			q_{n,1} = nF_{n-1};
			\]
		\item For $n \ge 4$,
			\[
			q_{n,2} = \frac 15nF_{n-3} + \frac 1{10}(n-3)n L_{n-2};
			\]
		\item For $n \ge 6$,
			\[
			q_{n,3}= \frac{2}{25}nF_{n-5} + \frac 1{25}(n-5)nL_{n-4} + \frac1{30}(n^2-9n+20)nF_{n-3}.
			\]
	\end{enumerate}
\end{theorem}

Let
\[
Q_n(x) = \sum_{k\ge 0}q_{n,k}x^k%
\]
be the cube polynomial of $\Omega_n$.
We list the first few of $Q_n(x)$ as follows.
\begin{align*}
  Q_0(x) &= 1 \\
  Q_1(x) &= 2+x \\
  Q_2(x) &= 3+2x \\
  Q_3(x) &= 4+3x \\
  Q_4(x) &= 7+8x+2x^2 \\
  Q_5(x) &= 11+15x+5x^2 
\end{align*}

It is easy to show the recurrence relation of $Q_n(x)$ from Proposition~\ref{prop:rec-Q}.
\begin{proposition}\label{prop:rec-Q}
	For $n\ge 4$,
	\[
	Q_n(x) = Q_{n-1}(x) + (1+x)Q_{n-2}(x).
	\]
\end{proposition}

From Proposition~\ref{prop:rec-Q} again, a result on $Q_n(x)$ is obtained easily.
\begin{theorem}
  For $n\ge 2$,
  \[
  Q_n(x) = \left(\frac{1+\sqrt{5+4x}}2\right)^n + \left(\frac{1-\sqrt{5+4x}}2\right)^n.
  \]
\end{theorem}

\begin{proposition}
	For $n \ge 2$, the root of $Q_n(x)$ is
	\[
	x_{n,k} = -\frac{5+\tan^2\frac{(2k-1)\pi}{2n}}4,
	\]
	where $k = 1,2,\dots,\lfloor n/2 \rfloor$.
\end{proposition}

Since the roots of a polynomial are all negative reals is log-concave and a positive log-concave sequence is unimodal \cite{bComte74,aStanl89,bWilf94}, the follow corollary is immediate.
\begin{corollary}
	For all $n \ge 2$, the sequences of coefficients of $Q_n(x)$ is log-concave and unimodal.
\end{corollary}

The Jacobsthal-Lucas numbers $J_n$ (see A014551 in \cite{Sloan19}) are defined by the recurrence relation: $J_0 = 2$, $J_1 = 1$ and $J_n = J_{n-1} + 2J_{n-2}$ for $n \ge 2$.
By Proposition~\ref{prop:rec-Q}, the relation of between $Q_n(x)$ and $L_n$ or $J_n$ is immediate.
For $n \ge 2$, $Q_n(0) = L_n$ and $Q_n(1) = J_n$.
In addition, the generating function of $Q_n(x)$ is proved.
\begin{theorem}\label{th:gf-Q}
	The generating function of $Q_n(x)$
	\[
	\sum_{n\ge 0}Q_n(x)y^n = \frac{2-y}{1-y-(x+1) y^2} + (1+x)y - 1.
	\]
\end{theorem}

The cube polynomial can be deduced from the generating function of $Q_n(x)$. 
\begin{proposition}
  For $n\ge 2$,
  \[
  Q_n(x) = \sum_{j\ge0}Y(n-j,j)(1+x)^j = \sum_{j\ge0}\left( \binom{n-j}j + \binom{n-j-1}{j-1} \right)(1+x)^j.
  \]
\end{proposition}

The number of induced $k$-dimensional cubes of $\Omega_n$ is follows immediately.
\begin{corollary}\label{cor:bin-q}
For $n\ge 2$,
\[
q_{n,k} = \sum_{j=k}^{\lfloor n/2 \rfloor}Y(n-j,j)\binom jk = \sum_{j=k}^{\lfloor n/2 \rfloor}\left( \binom{n-j}j + \binom{n-j-1}{j-1} \right)\binom jk.
\]
\end{corollary}

\textbf{Remark}: In Klav\v{z}ar and Mollard \cite{aKlavM12}, $n-a-1$ is printed incorrectly as $n-a+1$ in Theorem~3.2 and $n-i-1$ as $n-i+1$ in Corollary~3.3.

By Corollary~\ref{cor:bin-q} and
\[
F_n = \sum_{j=0}^{\lfloor (n-1)/2 \rfloor} \binom{n-j-1}j = \sum_{j=1}^{\lfloor (n+1)/2 \rfloor} \binom{n-j}{j-1},
\]
we obtained again Part~(1) of Theorem~\ref{th:cenum4}.

As a consequence of Corollary~4.10 in \cite{aWangZY18}, it is clear that the follow corollary.
\begin{corollary}
	For $n \ge 2$,
	\[
	\sum_{k \ge 0}\sum_{j=k}^{\lfloor n/2 \rfloor} \left( \binom{n-j}j + \binom{n-j-1}{j-1} \right) \binom jk(-1)^k = 1.
	\]
	\[
	\sum_{k \ge 0}\sum_{j=k}^{\lfloor n/2 \rfloor} \left( \binom{n-j}j + \binom{n-j-1}{j-1} \right) \binom jk k(-1)^k = n.
	\]
\end{corollary}

Using Kronecker delta function again, we give another generating function.
\begin{theorem}
For $k \ge 1$ is fixed, the generating function of $q_{n,k}$ is
\[
\sum_{n\ge 0} q_{n,k} y^n = \frac{(2-y)y^{2k}}{(1-y-y^2)^{k+1}} + y\delta_{k1}.
\]
\end{theorem}

\subsection{Maximal cube polynomials}

Let $h_{n,k} := h_k(\Omega_n)$ be the number of maximal $k$-dimensional cubes in $\Omega_n$, let $h_{n,k} = 0$ if no maximal $k$-dimensional induced cube of $\Omega_n$ exists, and let $H_n(x) = \sum_{k \ge 0} h_{n,k} x^k$ be the maximal cube polynomial of $\Omega_n$.
Observe that although the recurrence relation is same as that of $g_{n,k}$ in \cite{aMolla12}, the initial values are different, our results are different from those in \cite{aMolla12}.

Combining $\Omega_{n-1} \subset \Omega_n$, Theorem~\ref{th:struc} and Figure~\ref{fig:struc} indicate that the recurrence relation of $h_{n,k}$.
\begin{proposition}\label{prop:rec-h}
	For $n \ge 4$,
	\[
	h_{n,k} = h_{n-2,k-1} + h_{n-3,k-1}.
	\]
\end{proposition}

We list the first few of $H_n(x)$ as follows.
\begin{align*}
	H_0(x) &= 1 \\
	H_1(x) &= x \\
	H_2(x) &= 2x \\
	H_3(x) &= 3x \\
	H_4(x) &= x+2x^2 \\
	H_5(x) &= 5x^2 
\end{align*}

And, from Proposition~\ref{prop:rec-h}, the recurrence relation of $H_n(x)$ is given easily.

\begin{proposition}\label{prop:rec-H}
	For $n \ge 5$,
	\[
	H_n(x) = xH_{n-2}(x) + xH_{n-3}(x).
	\]
\end{proposition}

The $(1,2,3)$-Padovan numbers $p'_n$ is defined as: $p'_0 = 1$, $p'_1 = 2$, $p'_2 = 3$ and $p'_n = p'_{n-2} + p'_{n-3}$ for $n \ge 3$.
It is not difficult to verify that for $n \ge 1$, $H_n(1) = p'_{n-1}$.

Furthermore, we obtain the generating functions of $H_n(x)$ and $p'_n$ by Proposition~\ref{prop:rec-H}.
\begin{theorem}\label{th:gf-H}
	The generating function of $H_n(x)$ is
	\[
	\sum_{n=0}^\infty H_n(x) y^n = \frac{2+y}{1 - xy^2(1+y)} - (1-x)y - 1;
	\]
	and the generating function of $p'_n$ is
	\[
	\sum_{n=0}^\infty p'_n y^n = \frac{1+2y+2y^2}{1-y^2-y^3}.
	\]
\end{theorem}

Expanding the right side of $\sum_{n \ge 0} H_n(x) y^n$ into formal power series, we have result of $h_{n,k}$ by Lucas triangle.
\begin{proposition}\label{prop:bin-Hh}
	For $n \ge 2$,
	\[
	h_{n,k} 
	 = \binom{k}{3k-n} + \binom{k+1}{3k-n+1} = Y(k+1,3k+1-n).
	\]
\end{proposition}

From another point of view, using Kronecker delta function, we get directly the generating function of $h_{n,k}$.
\begin{proposition}
	For $k \ge 1$,
	\[
	\sum_{n \ge 0} h_{n,k} y^n = \big(y^2(1+y)\big)^k(2+y) + y\delta_{k1}
	\]
\end{proposition}

And by Proposition~\ref{prop:bin-Hh}, the number of terms of $H_n(x)$ is determined easily.
\begin{proposition} 
	Let $n = 6m+b$ and $0 \le b \le 5$. The number of terms of $H_n(x)$ is
	\[
	\begin{cases}
		m+2, & \mbox{if } b = 4; \\
		m+1, & \mbox{otherwise}.
	\end{cases}
	\]
\end{proposition}

\subsection{Disjoint cube polynomials}

Let $s_{n,k} := s_k(\Omega_n)$ denote the maximum number of disjoint $k$-dimensional cubes in $\Omega_n$. Let $s_{n,k} = 0$ if no $k$-dimensional induced hypercubes of $\Omega_{n}$ exists. And let $S_n(x) = \sum_{k\ge0} s_{n,k}x^k$ be the disjoint cube polynomial.

Let
\[
\theta_n = 
\begin{cases}
0, & \mbox{if } 3\mid n, \\
1, & \mbox{otherwise}.
\end{cases}
\]
And let
\[
\eta_n=
\begin{cases}
-1, & \mbox{if } n = 0, \\
1, & \mbox{if } n = 1, \\
0, & \mbox{otherwise}.
\end{cases}
\]

Combining Theorem~\ref{th:struc} and Figure~\ref{fig:struc}, using $\theta_n$ and $\eta_n$, we obtain the recurrence relation of $s_{n,k}$ similar to $q_k(n)$ in \cite{aGraviMSZ15}.
The results in \cite{aGraviMSZ15,aSaygE16} are only on Fibonacci cubes, and there is no recurrence relations and generating functions in \cite{aSaygE16}.

\begin{proposition}\label{prop:rec-disq}
For $n\ge0$,
\[
s_{n,0} = \left\lceil \frac{L_{n}}2 \right\rceil = \frac{L_n + \theta_n}2,
\]
and
\[
s_{n,1} = \left\lfloor  \frac{q_{n,0}}2 \right\rfloor = \frac{L_{n}-\theta_n}2 + \eta_n.
\]

More general, for $n \ge 4$ and $k \ge 2$,
\[
s_{n,k} = s_{n-2,k-1} + s_{n-3,k}.
\]
\end{proposition}

We have the maximum number of maximum disjoint cubes by Propositions~\ref{prop:dimq} and~\ref{prop:rec-disq}.
\begin{corollary}
  The maximum number of maximum disjoint cubes is $1$ ($2\mid n$) or $2$ ($2\nmid n$) for $n\ge 2$.
\end{corollary}

The first few of $S_n(x)$ is listed as follows.
\begin{align*}
  S_0(x) &= 1 \\
  S_1(x) &= 1+x \\
  S_2(x) &= 2+x \\
  S_3(x) &= 2+2x \\
  S_4(x) &= 4+3x+x^2 \\
  S_5(x) &= 6+5x+2x^2 
\end{align*}

Note that $2\mid L_{3m}$ (for $m\ge0$), we have Theorem~\ref{th:rec-s} immediately by Proposition~\ref{prop:rec-disq}.
\begin{theorem}\label{th:rec-s}
For $n\ge 4$,
\[
S_n(x) = xS_{n-2}(x) + S_{n-3}(x) + \frac{L_{n-2}-\theta_n}2 x - \eta_{n-3} x + \frac{L_n-L_{n-3}}2.
\]
\end{theorem}

Combining Theorem~\ref{th:rec-s} and the generating functions of $L_n$, $\theta_n$ and $\eta_n$, the generating function of $S_n(x)$ can be obtained by a tedious calculations.
\begin{theorem}
	The generating function of $S_n(x)$ is
	\[
	\sum_{n\ge0}S_n(x)y^n = \frac{1-(3-x)y^3+(2-x)y^6+xy^8}{(1-y-y^2)(1-y^3)(1-xy^2-y^3)}+xy.
	\]
\end{theorem}

By the generating function of $\theta_n(x)$, we obtain $\sum_{n \ge 0} s_{n,1} y^n$; moreover, we have the following theorem.
\begin{theorem}
	For $k \ge 1$ is a fixed integer, the generating function of $s_{n,k}$ is
	\[
	\sum_{n \ge 0} s_{n,k} y^n = \frac12\left( \frac{y^2}{1-y^3} \right)^{k-1} \left(\frac{y+2y^2}{1-y-y^2}-\frac{y+y^2}{1-y^3} \right) + y\delta_{k1}.
	\]
\end{theorem}

\subsection{Degree sequences polynomials}

Let $d_{n,k} := d_k(\Omega_n)$ denote the number of vertices of degree $k$ in $\Omega_n$, i.e.\ $d_{n,k} = |\{\,v\in V(\Omega_n) \mid \operatorname{deg}_{\Omega_n}(v)=k\,\}|$, and let $d_{n,k} = 0$ if no vertex of degree $k$ in $\Omega_n$ exists.
In fact, although the recurrence relation is same as that of $\ell_{n,k}$ in \cite{aKlavzMP11}, the initial values are different, thus our results are different from those in \cite{aKlavzMP11}.
The following proposition can be shown by inducing on $|V(\Omega_n)|$.

\begin{proposition}
	The minimum and maximum degree of $\Omega_n$ are $\delta(\Omega_n) = \lfloor \frac {n+1}3 \rfloor$ and 
	$\Delta(\Omega_n) = n-1$ ($n\ge 3$), respectively. 
	Thus, the both connectivity and edge connectivity of $\Omega_n$ is $\lfloor \frac {n+1}3 \rfloor$ \cite{aZhangGC88}. 
	
	Moreover, for $m\ge 1$,
	\[
	d_{n,\delta(\Omega_n)} =
	\begin{cases}
	\frac{m(m+3)}2, & \mbox{if } n = 3m-1; \\
	m+1, & \mbox{if } n = 3m; \\
	1, & \mbox{if } n = 3m+1.
	\end{cases}
	\]
	And for $n \ge 5$, $d_{n,\Delta(\Omega_n)} = 2$.
\end{proposition}

By the convex expansion for finite distributive lattices in \cite{aWangZY18}, we have a lemma on degree.
\begin{lemma}[\cite{aWangZY18}]\label{lem:deg}
  Let $L$ be a finite distributive lattice. If $K$ is a cutting of $L$, then
  \[
  d_k((L \boxplus K )\boxplus K) = d_k(L \boxplus K) + d_{k-2}(K).
  \]
\end{lemma}

Combining Theorem~\ref{th:struc} and Lemma~\ref{lem:deg}, we get the relation of $\Omega_n$ and $\Gamma_n$ on degree.
\begin{proposition}
For $n\ge3$,
\[
d_{n,k} = d_k(\Gamma_{n-1}) + d_{k-2}(\Gamma_{n-3}).
\]
\end{proposition}

\begin{lemma}[\cite{aKlavzMP11}]
	The number of vertices of degree $k$ of $\Gamma_n$ is  
	\[
	d_k(\Gamma_n) = \sum_{j=0}^k \binom{n-2j}{k-j} \binom{j+1}{n-k-j+1}.
	\]
\end{lemma}

Hence, a formula of $d_{n,k}$ follows immediately.
\begin{proposition}\label{prop:bin-d}
	For $n \ge 2$, the number of vertices of degree $k$ of $\Omega_n$ is
	\[
	d_{n,k} = \sum_{j=0}^k \binom{j+1}{n-k-j} \left( \binom{n-2j-1}{k-j} + \binom{n-2j-3}{k-j-2} \right).
	\]
\end{proposition}

On the other hand, we have the recurrence relation of $d_{n,k}$ illustrated in Figure~\ref{fig:mlc-deg}.
\begin{proposition}\label{prop:rec-d}
For $n \ge 5$,
\[
d_{n,k} = d_{n-2,k-1} + d_{n-1,k-1} - d_{n-3,k-2} + d_{n-3,k-1}.
\]
\end{proposition}

\begin{figure}[!h]
	\centering
	\begin{tikzpicture}[scale=0.7]
		\draw[rotate=30] (0,0) ellipse (2cm and 1cm)
		(1.8,-2) ellipse (2cm and 1cm)
		(0.7,0.16) ellipse (1.224cm and 0.612cm)
		(1.62,1.58) ellipse (1.224cm and 0.612cm);
		\draw (-1.732+0.33,-1-0.26) -- (0.9+0.28,-1.92-0.18)
		(-1.732-0.3+3.34,-1+0.14+2.15+1.7) -- (-1.732-0.3+3.34,-1+0.14+2.15) -- (0.9-0.28+3.34,-1.92+0.2+2.15)
		(-1.732+0.15+3.34-2,-1-0.09+2.15+1.72-1.37) -- (-1.732+0.15+3.34-2,-1-0.09+2.15+1.72-1.38-1.71);
		\draw[dashed,rotate=30] (3.43,-0.5) ellipse (1.224cm and 0.612cm);
		\draw[dashed] (-1.732-0.3+3.34,-1+0.14+2.15+1.7) -- (0.9-0.15+3.34,-1.92+0.12+2.15+1.72)
		(-1.732+0.15+3.34-2,-1-0.09+2.15+1.72-1.38) -- (0.9+3.34-2,-1.92+2.15+1.72-1.38);
		
		\node at (0.6,2.1) {$\Omega_{n-3}$};
		\node at (2.5,-1) {$\Omega_{n-2}$};
		\node at (1.2,-2.5) {$\Omega_{n}$};
	\end{tikzpicture}
	\caption{Illustrating the recurrence relation of $d_{n,k}$}
	\label{fig:mlc-deg}
\end{figure}
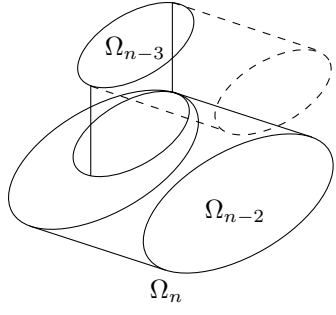

Let
\[
D_n(x) = \sum_{k\ge 0} d_{n,k} x^k
\]
be the degree sequence polynomial of $\Omega_n$ and we list the first few of $D_n(x)$.
\begin{align*}
  D_0(x) &= 1 \\
  D_1(x) &= 2x \\
  D_2(x) &= 2x+x^2 \\
  D_3(x) &= 2x+2x^2 \\
  D_4(x) &= x+3x^2+3x^3 \\
  D_5(x) &= 5x^2+4x^3+2x^4 
\end{align*}

In addition, the recurrence relation of $D_n(x)$ is a straightforward consequence.
\begin{proposition}\label{prop:rec-D}
For $n \ge 5$,
\[
D_n(x) = xD_{n-1}(x) + xD_{n-2}(x) + (x-x^2)D_{n-3}(x).
\]
\end{proposition}

Therefore, the generating function of $D_n(x)$ is obtained.
\begin{theorem}\label{th:gf-D}
	The generating function of $D_n(x)$ is given by 
	\[
	\sum_{n\ge 0}D_n(x)y^n = \frac{(1-xy +y)(1 + x^2y^2)}{(1 - xy)(1 - xy^2) - xy^3} + (2x-1)y.
	\]
\end{theorem}

Similar to \cite{aKlavzMP11}, from Theorem~\ref{th:gf-D}, using the expansion
\[
\frac{x^n}{(1-x)^{n+1}} = \sum_{j \ge n} \binom{j}{n} x^j,
\]
we give a proof of Proposition~\ref{prop:bin-d}.

\begin{proof}[of Proposition~\ref{prop:bin-d}]
	Consider the formal power series expansion of
	\[
	f(x,y) = \frac1{(1 - xy)(1 - xy^2) - xy^3},
	\]
	we obtain
	\[
	[x^k][y^n]f(x,y) = \sum_{j = 0}^k \binom{n-2j}{k-j} \binom j{n-k-j}.
	\]
	
	Note that
	\[
	\frac{1 +y -xy}{(1 - xy)(1 - xy^2) - xy^3} = f(x,y) + yf(x,y) - xyf(x,y),
	\]
	we have
	\begin{align*}
		&[x^k][y^n] \frac{1 +y -xy}{(1 - xy)(1 - xy^2) - xy^3} \\
		&= [x^k][y^n]F(x,y) + [x^k][y^{n-1}]F(x,y) - [x^{k-1}][y^{n-1}]F(x,y) \\
		&= \sum_{j = 0}^k \binom{n-2j}{k-j} \binom j{n-k-j} - \sum_{j = 0}^k \binom{n-2j-1}{k-j-1} \binom j{n-k-j} + \sum_{j = 0}^k \binom{n-2j-1}{k-j} \binom j{n-k-j-1} \\
		&= \sum_{j = 0}^k \binom{n-2j-1}{k-j} \binom j{n-k-j} + \sum_{j = 0}^k \binom{n-2j-1}{k-j} \binom j{n-k-j-1} \\
		&= \sum_{j = 0}^k \binom{n-2j-1}{k-j} \binom{j+1}{n-k-j}.
	\end{align*}
	
	In addition, we have
	\[
	[x^k][y^n] \frac{(1 +y -xy)x^2y^2}{(1 - xy)(1 - xy^2) - xy^3} = [x^{k-2}][y^{n-2}] \frac{1 +y -xy}{(1 - xy)(1 - xy^2) - xy^3} = \sum_{j = 0}^{k-2} \binom{n-2j-3}{k-j-2} \binom{j+1}{n-k-j}.
	\]
	
	Hence, Proposition~\ref{prop:bin-d} holds for $n \ge 2$.
\end{proof}

And since $d_{n,k} > 0$ for $k \in \{ \lfloor \frac{n+1}{3} \rfloor,\dots,n-1 \}$, the following corollary holds.
\begin{corollary}
	The degree spectrum of $\Omega_n$ is continuous.
\end{corollary}

\subsection{Indegree and outdegree sequence polynomial}

Let $d_{n,k}^-$ denote the number of vertices of indegree $k$ in $\Omega_n$, or the number of anti-chains with exactly $k$ elements in $\Xi_n$, or the number of elements covered exactly by $k$ elements in $\mathcal{F}(\Xi_n)$.
From the structure of $\Omega_n$ as shown in Figure~\ref{fig:struc} and Theorem~4.7 \cite{aWangZY18}, it is not difficult to obtain the recurrence relation of $d^-_{n,k}$.

\begin{proposition}\label{prop:rec-id}
	For $n\ge 0$, $d_{n,0}^- = 1$; and for $n \ge 4$, $k \ge 1$,
	\[
	d_{n,k}^- = d_{n-1,k}^- + d_{n-2,k-1}^-.
	\]
\end{proposition}

Moreover, we have the following corollary.
\begin{corollary}
	For $k \ge 0$ and $n \ge 2k+3$,
	\[
	d^-_{n,k} = \sum_{j=0}^k d^-_{n-2j-1,k-j}.
	\]
\end{corollary}

Let
\[
D_n^-(x) = \sum_{k=0}^{\lfloor \frac n2 \rfloor} d_{n,k}^-x^k
\]
be the indegree sequence polynomial of $\Omega_n$ and list the first few of $D_n^-(x)$ as follows.

\begin{align*}
	D_0^-(x) &= 1 \\
	D_1^-(x) &= 1+x \\
	D_2^-(x) &= 1+2x \\
	D_3^-(x) &= 1+3x \\
	D_4^-(x) &= 1+4x+2x^2 \\
	D_5^-(x) &= 1+5x+5x^2 
\end{align*}

It follows that the recurrence relation of $D^-_n(x)$ from Proposition~\ref{prop:rec-id}.
\begin{proposition}\label{prop:rec-ID}
	For $n \ge 4$,
	\[
	D_n^-(x) = D_{n-1}^-(x) + xD_{n-2}^-(x)
	\]
\end{proposition}

On the other hand, we have two propositions on the cube polynomials.
\begin{proposition}\label{prop:relde}
	By the relation of degree and edge in a graph, 
	\[
	\frac{\partial D_n(x)}{\partial x}\bigg|_{x=1} = 2\left.\frac{\partial D_n^-(x)}{\partial x}\right|_{x=1} = 2q_{n,1}.
	\]
\end{proposition}

\begin{proposition}[\cite{aWangZY18}]\label{prop:QD-}
	For $n \ge 0$,
	\[
	D_n^-(x) = Q_n(x-1).
	\]
\end{proposition}

Combining Proposition~\ref{prop:QD-} and results on cube polynomials, it is obvious that \ref{th:gf-D-}--\ref{cor:lcu-D-} hold.
\begin{theorem}\label{th:gf-D-}
	The generating function of $D_n^-(x)$ is
	\[
	\sum_{n\ge0} D_n^-(x)y^n = \frac{2-y}{1-y-xy^2} + xy-1.
	\]
\end{theorem}

\begin{proposition}
	For $n \ge 2$,
	\[
	d_{n,k}^- = Y(n-k,k) = \binom{n-k}k + \binom{n-k-1}{k-1}.
	\]
\end{proposition}

\begin{corollary}
	The maximum indegree of $\Omega_n$ is $\lfloor n/2 \rfloor$, and the indegree spectrum of $\Omega_n$ is continuous.
\end{corollary}

We can obtain another formula on $D_n^-(x)$ by Proposition~\ref{prop:rec-ID}.
\begin{proposition}
	For $n \ge 1$,
	\[
	D_n^-(x) = \left(\frac{1+\sqrt{1+4x}}2\right)^n + \left(\frac{1-\sqrt{1+4x}}2\right)^n
	\]	
\end{proposition}

\begin{theorem}
	The roots of $D_n^-(x)$ is 
	\[
	x=-\frac{1+\tan^2\frac{(2k-1)\pi}{2n}}4,
	\]
	for $n \ge 2$, where $1 \le k \le n$.
\end{theorem}

\begin{corollary}\label{cor:lcu-D-}
	For all $n \ge 2$, the sequences of coefficients of $D_n^-(x)$ is log-concave and unimodal.
\end{corollary}

We also obtain another generating function of $d_{n,k}^-$.

\begin{theorem}
	For $k \ge 1$ is fixed, the generating function of $d_{n,k}^-$ is
	\[
	\sum_{n \ge 0} d_{n,k}^- y^n = \frac{y^{2k}(2-y)}{(1-y)^{k+1}} + y\delta_{k1}.
	\]
\end{theorem}

The results on outdegree are exactly same as indegree \cite{aWangZY18}, thus they are not listed here.

\section{Summary}
By the consequences on $\Lambda_n$ \cite{aZhangYY14}, We have the following relation of the matchable Lucas cube $\Omega_n$ and the Lucas cube $\Lambda_n$ for $n\ge 2$.
\begin{theorem}
	The number of induced cubes of $\Lambda_n$ and $\Omega_n$ is same, but structure is different.
	In other word, for $m \ge 1$, $\Lambda_{2m}$ can be orientated as the Hasse digram of a finite distributive lattice that is not isomorphic to $\Omega_{2m}$, but $\Lambda_{2m+1}$ can not be a Hasse digram of a finite distributive lattice.
\end{theorem}

\section*{Acknowledgments}
\addcontentsline{toc}{section}{Acknowledgments}

The authors are grateful to the referees for their careful reading and many valuable suggestions.


\end{document}